%% file: ratioRevised1_arxiv.tex
\newcommand{\comment}[1]{}

\documentclass{article}
\usepackage{natbib}
\usepackage{graphicx} 
\usepackage[all]{xy} 
\usepackage{amsfonts, amssymb, amsmath, url} 
\usepackage{epsfig}
\usepackage{epstopdf}

\usepackage{ifthen}
\newboolean{format}
\setboolean{format}{false}

\ifthenelse{\boolean{format}}{
\smartqed 
\spnewtheorem{Def}{Definition}{\bf}{}
\spnewtheorem{Rem}{Remark}{\bf}{}
\spnewtheorem{Ques}{Question}{\bf}{}
\spnewtheorem{Example}{Example}{\bf}{}
}{
\usepackage{amsthm}\theoremstyle{definition}

\newtheorem{Rem}{Remark}
\newtheorem{Ques}{Question}
\newtheorem{Example}{Example}
}
\newtheorem{Thm}{Theorem}

\newtheorem{Prop}{Proposition}
\newtheorem{Coro}{Corollary}

\newcommand{\Einttt}{\int^{\!\!\!\!\!\!\!\!\!e}}

\def\qqed{\ifthenelse{\boolean{format}}{\qed}{}}

\include{include_macros}

\title{On the Regular Variation of Ratios of Jointly Fr\'echet Random Variables}

\begin{document}\sloppy

\ifthenelse{\boolean{format}}{
\date{February, 2011}
\journalname{Extremes}
\author{Yizao Wang}
\titlerunning{Regular Variation of Ratios}
\authorrunning{Yizao Wang}
\institute{Yizao Wang \at
Department of Statistics, the University of Michigan\\
              439 West Hall, 1085 South University, Ann Arbor, MI, 48109--1107 \\
              \email{yizwang@umich.edu}
}}{\author{Yizao Wang \\Department of Statistics, the University of Michigan}}
\maketitle
\begin{abstract}
We provide a necessary and sufficient condition for the ratio of two jointly $\alpha$-Fr\'echet random variables to be regularly varying. This condition is based on the spectral representation of the joint distribution and is easy to check in practice. Our result motivates the notion of the {\it ratio tail index}, which quantifies dependence features that are not characterized by the {\it tail dependence index}. As an application, we derive the asymptotic behavior of the quotient correlation coefficient proposed in \cite{zhang08quotient} in the dependent case. Our result also serves as an example of a new type of regular variation of products, different from the ones investigated by~\cite{maulik02asymptotic}.

\ifthenelse{\boolean{format}}{
}
{
\vspace{0.1in}\noindent{\bf Keywords:} Regular variation, multivariate Fr\'echet distribution, spectral representation, tail dependence index, logistic model, mixed model

}
\end{abstract}
\section{Introduction}
Regular variation is often used to describe the tail behavior of random variables.
A measurable function $U:\mathbb R_+\to \mathbb R_+$ is {\it regularly varying} at infinity with index $\rho$, written $U\in RV_{-\rho}$, if 
\[
\limt \frac{U(tx)}{U(t)} = x^{-\rho}\,,\mfa x>0\,.
\]
When $\rho = 0$ we say $U$ is slowly varying. A random variable $Z$ is regularly varying with index $\rho\geq0$, written $Z\in RV_{-\rho}$, if the tail distribution $\wb F(t) = \proba(Z>t)$ is regularly varying: $\wb F\in RV_{-\rho}$ (see e.g.~\cite{resnick87extreme}). For the sake of simplicity, we only consider non-negative random variables. 

In extreme value theory, the notion of regular variation plays an important role in characterizing the domain of attraction of random variables. Namely, a random variable $Z$ is regularly varying with index $\alpha>0$, if and only if 
\[
\frac1{\kappa_n}\bveein Z_i\weakto \zeta_\alpha
\]
as $n\to\infty$, where $\indn \kappa$ is some normalizing sequence, $Z_1,Z_2,\dots$ are independent and identically distributed (i.i.d.) copies  of $Z$, `$\weakto$' stands for convergence in distribution and $\zeta_\alpha$ is a standard $\alpha$-Fr\'echet random variable with distribution
\[
\Phi_\alpha(t) = \proba(\zeta_\alpha\leq t) = \exp(-t^{-\alpha})\,, t>0\,.
\]
See e.g.~\cite{resnick87extreme}, Proposition 1.11.
In this case, the random variable $Z$ is said to be in the domain of attraction of $\Phi_\alpha$. The notion of regular variation for multivariate random vectors and stochastic processes have also been investigated extensively. See e.g.~\cite{resnick87extreme,resnick07heavy}, \cite{dehaan06extreme}, \cite{balkema07high}, and \cite{hult05extremal}, among others. Other applications of regular variation include, just to mention a few, the domain of attraction of partial sums of i.i.d.~random vectors (\cite{rvaceva62domains}), large deviations of regularly varying random walks (\cite{hult05functional}) and finite-dimensional distributions of the stationary solution of stochastic recurrence equations (\cite{kesten73random} and \cite{goldie91implicit}).

In this paper, we consider the regular variation of ratios of two random variables $X$ and $Y$. From a practical point of view, the ratio $X/Y$ can be seen as a random normalization of $X$ by $Y$. In extreme value theory, when modeling extremal behaviors, certain normalizations (thresholding) of values are often required. Random normalization sometimes has appealing theoretical properties and simplifies the statistical applications (see e.g.~\cite{heffernan07limit}, Section 4). 

By viewing $X/Y$ as the product of $X$ and $1/Y$, the problem is closely related to the regular variation of products of random variables. When the two random variables are independent, this problem has been addressed by \cite{breiman65some}. On the other hand, \cite{maulik02asymptotic} investigated certain dependence cases, which were then applied to the modeling of network traffic.

We address the regular variation of ratios in a specific case. Namely, we suppose that $(X,Y)$ has a {\it bivariate $\alpha$-Fr\'echet distribution} (or, $(X,Y)$ are {\it jointly $\alpha$-Fr\'echet}), i.e., for all $a,b\geq 0$, $\max(aX,bY)$ has $\alpha$-Fr\'echet distribution. To study the bivariate $\alpha$-Fr\'echet distributions, an efficient tool is the {\it spectral representation} introduced by~\cite{dehaan84spectral} and developed by~\cite{stoev06extremal} (a brief review will be given in Section~\ref{sec:preliminaries}). Based on the spectral representation, we provide a necessary and sufficient condition for $X/Y$ to be regularly varying (Theorem~\ref{thm:RV}). If this is the case, then the regular variation index is referred to as the {\it ratio tail index} of $X/Y$. We demonstrate that our condition is easy to check through a few popular models.

Our specific setting provides examples that are not covered by the results in~\cite{breiman65some} and~\cite{maulik02asymptotic}. Furthermore, we show that the ratio tail index does not characterize the dependence between $X$ and $Y$ in the traditional sense. We will compare the ratio tail index and the {\it tail dependence index} (see e.g.~\cite{sibuya60bivariate}, \cite{dehaan77limit} and \cite{ledford96statistics}), which has been widely used to quantify asymptotic (in)dependence of random variables.

As the main application of our result, we derive the asymptotic behavior of the {\it quotient correlation coefficient} (Theorem~\ref{thm:1}) for jointly Fr\'echet distributions. This coefficient was proposed by~\cite{zhang08quotient} and applied to test the independence of two random variables $X$ and $Y$. In this so-called {\it gamma test}, the quotient correlation coefficient is based on the independent samples of ratios $X/Y$ and $Y/X$. The asymptotic behavior of the coefficient has so far been studied only in the case when $X$ and $Y$ are independent Fr\'echet. 

Our result provides new theoretical support for the gamma test. We show that, when $(X,Y)$ is bivariate Fr\'echet, the power of the gamma test is high most of the time. Indeed, the asymptotic behavior of the quotient correlation coefficient is essentially determined by the ratio tail indices of $(X,Y)$ and $(Y,X)$, if they exist. Furthermore, if the ratios have lighter tails than the single variables, then the gamma test rejects the null hypothesis with probability going to one as the sample size increases to infinity (Corollary~\ref{coro:power}). We also show that, when the ratios have tails equivalent to the ones of marginals, then in a `worst scenario', the gamma test performs also reasonably well (Example~\ref{example:power}).

The paper is organized as follows. Section~\ref{sec:preliminaries} provides preliminaries on the regular variation and the spectral representations of bivariate $\alpha$-Fr\'echet random vectors. Section~\ref{sec:ratio} is devoted to the characterization of regular variation of $X/Y$, based on the joint distribution of $(X/Y,Y)$. 
Section~\ref{sec:examples} calculates the ratio tail indices for several well-known examples. 
In Section~\ref{sec:quotientCorrelation}, we review Zhang's gamma test and apply the result in Section~\ref{sec:ratio} to derive the asymptotic behavior of the quotient correlation coefficient in the dependent case. Section~\ref{sec:proofs} provides a proof of the joint distribution of $(X/Y,Y)$, based on a result by~\cite{weintraub91sample}. Finally, in Section~\ref{sec:discussion} we briefly discuss the connection between our results and some related works.

\section{Preliminaries}\label{sec:preliminaries}
In this section, we review the spectral representation of bivariate $\alpha$-Fr\'echet distributions. We also introduce some notations that will be used in the rest of the paper.

We will focus on bivariate 1-Fr\'echet random vector $(X,Y)$. 
Every bivariate 1-Fr\'echet random vector $(X,Y)$ has the following spectral representation:
\equh\label{eq:XY}
(X,Y)\eqd \bpp{\Einttt_Sf(s)M(\d s), \Einttt_Sg(s)M(\d s)}\,.
\eque
Here, `$\ \Einttt\ \ $' stands for the extremal integral, $(S,\calB_S,\mu)$ is a standard Lebesgue space, for example, a Polish space $(S,\rho)$ with a $\sigma$-finite measure $\mu$ on its Borel sets $\calB_S$, $f,g$ are measurable non-negative and integrable functions on $(S,\calB_S,\mu)$, and $M$ is a 1-Fr\'echet random sup-measure on $(S,\calB_S)$ with control measure $\mu$ (see e.g.~\cite{dehaan84spectral} and \cite{stoev06extremal}). The functions $f$ and $g$ are called the {\it spectral functions} of $X$ and $Y$, respectively, and the joint distribution of $(X,Y)$ can be expressed as follows:
\equh\label{eq:fdd0}
\proba(X\leq x,Y\leq y) = \exp\bccbb{-\int_S{\frac{f(s)}x\vee \frac{g(s)}y}\mu(\d s)}\,.
\eque
The random vector in~\eqref{eq:XY} is said to be {\it standard}, if it has standard 1-Fr\'echet marginals or equivalently, $\int_S f\d\mu = \int_S g\d\mu = 1$. It is well known that any bivariate $\alpha$-Fr\'echet random vector can be easily transformed into a standard 1-Fr\'echet random vector (see e.g.~\cite{stoev06extremal}, Proposition 2.9).

For the spectral representation of standard bivariate 1-Fr\'echet distribution~\eqref{eq:XY}, specific choices of $(f,g)$ and $(S,\mu)$ often appear in the literature (see e.g.~\cite{dehaan06extreme}, Theorem 6.1.14). We will provide examples using the following one for our convenience. Set $f(s) = 2s, g(s) = 2(1-s)$ for $s\in S = [0,1]$, and $\mu = H$, a probability distribution on $[0,1]$ with mean $1/2$. In this case,
\equh\label{eq:01}
\proba(X\leq x,Y\leq y) = \exp\bpp{-\int_0^{1}\frac{2s}x\vee\frac{2(1-s)}yH(\d s)}\,.
\eque
Many of our examples are constructed by choosing a specific $H$, and their joint cumulative distribution functions often do not have simple forms. 

The bivariate $\alpha$-Fr\'echet distributions arise as limits of i.i.d.~bivariate regularly varying random vectors, and it is often convenient to use the notion of {\it vague convergence}, denoted by `$\vagueto$' (see e.g.~\cite{kallenberg86random} and~\cite{resnick87extreme}). In particular, a random vector $(V,W)$ is said to be regularly varying with index $\alpha$, if 
\equh\label{eq:VC}
t\proba\bbb{\bpp{\frac V{\kappa_t},\frac W{\kappa_t}}\in\cdot}\Vagueto \nu(\cdot)
\eque
in $M_+(\mathbb E)$. 
Here, $\kappa_t\in RV_{1/\alpha}$, $\mathbb E = [0,\infty]^2\setminus\{(0,0)\}$, $M_+(\mathbb E)$ is the space of all nonnegative Radon measures on $\mathbb E$, the limit $\nu\in\mathbb E$ is non-zero and $\nu(c\,\cdot) = c^{-\alpha}\nu(\cdot)$, for all $c>0$. The measure $\nu$ is called the {\it exponent measure}, and it is said to have non-degenerate marginals, if $\nu((x,\infty]\times[0,\infty])$ and $\nu([0,\infty]\times(y,\infty])$ are non-degenerate in $x$ and $y$, respectively.
By Proposition 5.11 in \cite{resnick87extreme}, when $\alpha = 1$ and $\nu$ has non-degenerate normalized marginals, one can write
\[
\nu\bccbb{([0,x]\times[0,y])^c} = \int_0^1\frac{2w}x\vee\frac{2(1-w)}y H(\d w)
\]
such that $\int_0^1H(\d w) = \int_0 2wH(\d w) = 1$. In this case, letting $(V_i,W_i), i=1,\dots,n$ be i.i.d.~copies of $(V,W)$, the vague convergence~\eqref{eq:VC} is equivalent to
\equh\label{eq:VW}
\bpp{\frac1{\kappa_n}\bveein V_i, \frac1{\kappa_n}\bveein W_i} \weakto (X,Y)
\eque
with $(X,Y)$ defined as in~\eqref{eq:01}.

For any random vector $(V,W)$, we say that $V$ and $W$ are {\it asymptotically independent}, if~\eqref{eq:VW} holds with independent $X$ and $Y$. This corresponds to the case when $f$ and $g$ have disjoint supports in~\eqref{eq:XY}, or $H = (\delta_{\{0\}}+\delta_{\{1\}})/2$ in~\eqref{eq:01} (after normalization), or  the exponent measure $\nu$ in~\eqref{eq:VC} concentrates on the two axes $\{(x,0), x>0\}\cup\{(0,y),y>0\}$ with non-degenerate marginals. 

Throughout this paper, for all measurable functions $f$ and measurable sets $S_0\subset S$, we write $\nn f_{S_0} = \int_{S_0}|f(s)|\mu(\d s)$, and define
\equh\label{eq:Dt}
D_t\defe\{s\in S:f(s)/g(s) \leq t\}\qmand E_t\defe  S\setminus D_t\,.
\eque
Here and in the sequel, we will use the convention $1/0 = \infty$.
For any standard bivariate 1-Fr\'echet random vector $(X,Y)$ as in~\eqref{eq:XY}, two important identities are
\[
\limt\nn f_{E_t} = \int_Sf\ind_{\{g=0\}}\d\mu\qmand\limt \nn g_{D_t} = \nn g_S = 1\,,
\]
which follow from the dominated convergence theorem. Furthermore, the joint distribution~\eqref{eq:fdd0} can be expressed as
\equh\label{eq:fdd}
\proba(X\leq x,Y\leq y) = \exp\bbb{-\bpp{x\inv{\nn f_{E_{x/y}}} + y\inv\nn g_{D_{x/y}}}}\,.
\eque
Finally, for any random variable $Z$ and $u\geq 0$, we consider a thresholded version of $Z$ denoted by
\equh\label{eq:u}
 Z(u) = u+(Z-u)_+ = \max(Z,u)
\eque
as in \cite{zhang08quotient}. We consider such thresholded random variables in order to analyze the quotient correlation coefficient (\eqref{eq:qun} below), although it will turn out that the threshold value $u$ does not play an essential role in the asymptotic behavior (see Remark~\ref{rem:CLL} below). 
For a random vector $(X,Y)$ in $\mathbb R^2$, let $\{(X_i,Y_i)\}_{i=1,\dots,n}$ denote $n$ i.i.d.~copies of $(X,Y)$.

\section{Ratios of Fr\'echet Random Variables}\label{sec:ratio}
In this section, we provide an explicit formula for the distribution of $X(u)/Y(u)$, $u\geq 0$. Here $X(u) = X$ if $u = 0$. This leads to a necessary and sufficient condition for $X(u)/Y(u)$ to be regularly varying. In the sequel, all the bivariate 1-Fr\'echet random vectors $(X,Y)$ with representation~\eqref{eq:XY} are assumed to be standard. We write $a_t\sim b_t$ if $\limt a_t/b_t = 1$. 

Our first result is an explicit formula for the joint distribution of $(X/Y,Y)$.
\begin{Prop}\label{prop:1}
Consider $(X,Y)$ as in~\eqref{eq:XY}. Then, for all $t\geq0, u\geq 0$, 
\equh\label{eq:prop1}
\proba(X/Y\leq t,Y>u)
= \bpp{1+\frac {\nn f_{E_t}}{t\nn g_{D_t}}}\inv\bbb{1-\exp\bpp{-\frac{\nn g_{D_t}+t\inv{\nn f_{E_t}}}u}}\,.
\eque
In particular, when $\mu(E_t) = 0$, we have $\proba(X/Y>t) = 0$.
\end{Prop}
The proof borrows a result from~\cite{weintraub91sample}, and is deferred to Section~\ref{sec:proofs}. 
\begin{Rem}
Proposition~\ref{prop:1} can be seen as a special case of the conditional limit law established in~\cite{heffernan07limit} (Propositions 4 and 5 therein), where $(X,Y)$ are assumed to satisfy certain regular-variation type condition. Their results describe the asymptotic limit of $t\proba[(X/Y,Y/t)\in\cdot]$ as $t\to\infty$. Here, thanks to the spectral representation, we provide an explicit formula for the joint distribution of $(X/Y, Y)$. In particular, our result readily implies that 
\eqnh
\lim_{u\to\infty}\proba(X/Y\leq t\mid Y>u) & = & 
\lim_{u\to\infty}\frac{\proba(X/Y\leq t, Y>u)}{\proba(Y>u)}
\\
& = & \lim_{u\to\infty}\bpp{1+\frac {\nn f_{E_t}}{t\nn g_{D_t}}}\inv\frac{1-\exp[-(\nn g_{D_t}+t\inv{\nn f_{E_t}})u\inv]}{1-\exp(-u\inv)}\\
& = & \nn g_{D_t} = \int_{S}g\ind_{\{f\leq tg\}}\d\mu\,,
\eqne
which recovers Equation (32) in~\cite{heffernan07limit}. For a more general and geometric treatment of the conditional limit law, see~\cite{balkema07high}.
\end{Rem}
Now, the distribution of $X(u)/Y(u)$ follows as a corollary.
\begin{Coro}
Consider $(X,Y)$ as in~\eqref{eq:XY}. Then, for all $t\geq1, u\geq 0$, 
\equh\label{eq:Xu/Yu}
\proba\bpp{\frac{X(u)}{Y(u)}>t} = \bpp{1+\frac{\nn g_{D_t}}{\nn f_{E_t}}t}\inv\bbb{1-\exp\bpp{-\frac{\nn g_{D_t}+t\inv{\nn f_{E_t}}}u}}\,.
\eque
In particular, for all $u\geq 0$, as $t\to\infty$,
\equh\label{eq:Xu/Yu2}
\proba\bpp{\frac{X(u)}{Y(u)}>t} \sim \frac{\nn f_{E_t}}t\bb{1-\exp(-u\inv)}\,.
\eque
\end{Coro}
\begin{proof}
Observe that for $t\geq1$, 
\eqnh
& & 
\proba\bpp{\frac{X(u)}{Y(u)}>t} = \proba(X>t(u+(Y-u)^+), X\geq u) \\
\eqnhspace \quad = \proba(X>tu, Y\leq u) + \proba(X>tY, Y>u)\\
\eqnhspace \quad = \proba(Y\leq u) - \proba(X\leq tu, Y\leq u) +\proba(Y>u) - \proba(X\leq tY, Y>u)\\
\eqnhspace \quad = 1 - \proba(X\leq tu, Y\leq u) - \proba(X\leq tY, Y>u)\,.
\eqne
Plugging in~\eqref{eq:fdd} and~\eqref{eq:prop1} yields~\eqref{eq:Xu/Yu}, whence~\eqref{eq:Xu/Yu2} follows immediately, noting $\limt \nn g_{D_t} + t\inv\nn f_{E_t} = \nn g_S = 1$.\qqed
\end{proof}
Inspired by~\eqref{eq:Xu/Yu2}, define 
\equh\label{eq:gamma}
\gamma(t) \defe \frac{\nn f_{E_t}}t = \frac1t{\int_Sf\ind_{\{f>tg\}}\d\mu}\,.
\eque
Clearly, $\limt \gamma(t) = 0$. Then, provided $\gamma(t)>0$, or equivalently, $\mu(E_t)>0$ for all $t\in(0,\infty)$,
~\eqref{eq:Xu/Yu2} implies that for all $x>0$, 
\[
\frac{\proba( X(u)/ Y(u)>tx)}{\proba( X(u)/ Y(u)>x)}
\sim \frac{\gamma(tx)/(1+\gamma(tx))}{\gamma(x)/(1+\gamma(x))} \sim \frac{\gamma(tx)}{\gamma(x)}\,, \mmas t\to\infty\,.
\]
Therefore, we have thus proved the following result.
\begin{Thm}\label{thm:RV}
Consider $(X,Y)$ given in~\eqref{eq:XY} and suppose $\mu(E_t)>0$ for all $t\in(0,\infty)$. Then, 
\equh\label{eq:RV}
X(u)/Y(u)\in RV_{-\alpha} \mbox{ for some (all) } u\geq 0\mbox{, if and only if } \gamma\in RV_{-\alpha}\,.
\eque
\end{Thm}
In other words, for bivariate 1-Fr\'echet random vector $(X,Y)$, to study the regular variation of $X/Y$, it is equivalent to study the regular variation of $\gamma(t)$ in~\eqref{eq:gamma}, based on the spectral functions. We will see in the next section that for many well-known examples, the regular variation of $\gamma(t)$ follows from simple calculations. 
\begin{Rem} 
When $\mu(E_t) = 0$, or equivalently $\gamma(t) = 0$ for $t$ large enough, we have $\proba(X(u)/Y(u)>t)\leq \proba(X/Y>t) = 0$ by Proposition~\ref{prop:1}. 
This situation is relatively simple, and we do not study this case in this paper.
\end{Rem}
\begin{Rem}
A more general setting should be to consider $(X,1/Y)$ in some domain of attractions, such that $X/Y\in RV_{-\alpha}$. \cite{maulik02asymptotic} investigated certain general cases under this framework. Our specific case, however, is not covered by their results. See more discussion in Section~\ref{sec:discussion}. 
\end{Rem}
From now on, we say $(X,Y)$ has {\it ratio tail index} $\alpha$, if~\eqref{eq:RV} holds. We immediately have the following consequences. 
In the sequel, for any non-decreasing function $U$ on $\mathbb R$, let $U^\leftarrow(y)\defe \inf\{s:U(s)\geq y\}$ denote the left-continuous inverse of $U$.
\begin{Coro}\label{coro:2}
Consider $(X,Y)$ given in~\eqref{eq:XY} and suppose $\mu(E_t)>0$ for all $t\in (0,\infty)$. Suppose that $(X,Y)$ has ratio tail index $\alpha$, then the following statements hold:
\begin{itemize}
\item[(i)] $\alpha = 1$ if and only if $\nn f_{E_t}$ is slowly varying,
\item[(ii)] for all $u\geq 0$ and for all sequences $\kappa_n\sim (1/\gamma)^\leftarrow(Cn)$, 
\equh\label{eq:convergence}
\frac1{\kappa_n}\bveein\frac{X_i(u)}{Y_i(u)}\weakto \zeta_\alpha\,,
\eque
where $C = 1-\exp(-u\inv)$ and $\zeta_\alpha$ is a standard $\alpha$-Fr\'echet random variables, and 
\item[(iii)] for any sequence $\{u_n\}_{n\in\mathbb N}$ such that $u_n = o(n)$, the convergence in~\eqref{eq:convergence} holds with $u$ replaced by $u_n$ and $\kappa_n\sim(1/\gamma)^\leftarrow(n/u_n)$. 
\end{itemize}
\end{Coro}
\begin{proof}
Part (i) is trivial. Part (ii) follows from Proposition 1.11 in \cite{resnick87extreme} and part (iii) can be proved by a similar argument.\qqed
\end{proof}
An important consequence of Theorem~\ref{thm:RV} is that, the tail of the ratio is always lighter than or equivalent to the tails of $X$ and $Y$. Indeed,~\eqref{eq:gamma} implies $\gamma(t) = O(t\inv)$ as $t\to\infty$, thus $\gamma\in RV_{-\alpha}$ implies $\alpha\geq 1$. 
The case $\alpha = 1$ includes the case when $X$ and $Y$ are independent Fr\'echet random variables. This follows from~\cite{breiman65some}: if $X\in RV_{-1}$, $Z$ has finite $1+\epsilon$ moment for some $\epsilon>0$, and $X$ and $Z$ are independent, then $XZ\in RV_{-1}$ (here, $Z=1/Y$ has standard exponential distribution).
 
However, for dependent $X$ and $Y$, the ratio tail index $\alpha$ can still equal 1. By Part (i) of Corollary~\ref{coro:2}, a simple case of $\alpha=1$ is when $\limt \nn f_{E_t} = \int_Sf\ind_{\{g =0\}}\d\mu>0$. This is the case, in view of the spectral representation~\eqref{eq:01}, when there is a point mass at $\{1\}$, or equivalently when one can write $X = X_1\vee X_2$, such that $X_1$ is 1-Fr\'echet or equal to 0 almost surely, $X_2$ is 1-Fr\'echet, and $X_2$ is independent of $X_1$ and $Y$. If in addition $X_1 = 0$, then $X$ and $Y$ are independent.
We can also have examples such that $\alpha=1$ and $\int_Sf\ind_{\{g=0\}}\d\mu = 0$, by constructing $\nn f_{E_t}$ to be slowly varying, in a similar way as in Example~\ref{example:notRV} below.

Next, we compare the ratio tail index and the {\it tail dependence index} defined by
\equh\label{eq:TDI}
\lambda = \limt\proba(X>t\mid Y>t)\,,
\eque
provided the limit exists. The random variables $X$ and $Y$ are {asymptotically independent} if $\lambda = 0$, and {asymptotically dependent} if $\lambda\in(0,1]$. The tail dependence index has been widely studied and applied. See e.g.~\cite{sibuya60bivariate}, \cite{dehaan77limit}, \cite{ledford96statistics}, among others. 

It is easy to see that when $(X,Y)$ is bivariate 1-Fr\'echet, $X$ and $Y$ are asymptotically independent if and only if they are independent, as $(X,Y)$ is in its own domain of attraction. 
As a consequence, the tail dependence index (when $\lambda>0$) measures the dependence strength of bivariate 1-Fr\'echet random vectors. 
To understand the difference between the tail dependence index and the ratio tail index, observe that in our case, 
\[
\limt\proba(X>t\,|\, Y>t) = \nn{\min(f,g)}_S = \int_0^12\min(s,1-s)H(\d s)\,.
\]
The tail dependence index is determined by the spectral measure $H$ on $(0,1)$ (excluding $\{0,1\}$), while the ratio tail index is determined by the behavior of $H$ on a small neighborhood of $\{0,1\}$. 
\begin{Example}\label{example:patho}
Consider $(X,Y)$ given by
\equh\label{eq:rho}
(X,Y) \eqd ((1-\rho)Z_2\vee \rho Z_3,\rho Z_1\vee (1-\rho)Z_2)\,,
\eque
where $\rho\in[0,1]$ and $Z_1, Z_2$ and $Z_3$ are independent standard 1-Fr\'echet random variables. This corresponds, in~\eqref{eq:01}, to choose 
\[
H = \frac\rho2(\delta_{\{0\}}+\delta_{\{1\}})+(1-\rho)\delta_{\{1/2\}}\,,
\]
where $\delta_{\{a\}}$ is the unit point mass at $a$. It is easy to see $\gamma(t) = \rho/t$ for $t>1$. Therefore, $X/Y\in RV_{-1}$, that is, the ratio tail index always equals 1 for $\rho>0$. 

Observe that in this model, however, the dependence strength between $X$ and $Y$ varies for different choices of $\rho$. Observe that $\proba(X = Y) = \proba(Z_2\geq \rho/(1-\rho) (Z_1\vee Z_3)) = (1-\rho)/(1+\rho)$, which increases from 0 to 1 as $\rho$ decreases from 1 to 0. By direct calculation, the tail dependence index equals $\lambda = 1-\rho$, which reflects the dependence strength. The ratio tail index does not.
\end{Example}

We conclude this section with a remark on role of the threshold value $u$. The ratio tail indices of several concrete bivariate distributions are calculated in the following section.
\begin{Rem}\label{rem:CLL}
Theorem~\ref{thm:RV} and Corollary~\ref{coro:2} indicate that the threshold value $u$ does not play an essential role in the asymptotic behavior of $X(u)/Y(u)$. Indeed, from~\eqref{eq:convergence}, the limit, if it exists, is independent of $u>0$, up to multiplicative constants. This suggests that letting $u\to\infty$ is more interesting as $n\to\infty$, and part (iii) of Corollary~\ref{coro:2} shows that this would change the convergence rate, but the limit is still the same.

\end{Rem}
\section{Examples}\label{sec:examples}
In this section, we provide several examples on the ratio tail index. The first example is a concrete one where $X/Y$ is not regularly varying.
\begin{Example}\label{example:notRV}
We construct an example based on the spectral representation~\eqref{eq:01}. We will find an $H$ such that $\gamma(t)\sim \exp(-t)$ as $t\to\infty$. Here $\exp(-t)$ is not regularly varying and therefore, by Theorem~\ref{thm:RV}, $X/Y$ is not regularly varying. 

Suppose $H$ has Radon--Nikodym derivative $h(s)$ w.r.t.~the Lebesgue measure on $[0,1]$. First, we need to find $h(s)$ such that 
\[
\int_{\frac t{1+t}}^12sh(s)\d s = Ct\exp(-t)
\]
for $t$ large enough and some constant $C$ to be chosen later. Solving the above equation, we obtain
\[
h_1(s) = \frac C{2s(1-s)^3}\exp\bpp{-\frac s{1-s}}\,.
\]
One can choose a function $h_0:[0,1/2]\to\mathbb R$ and tune $C$ properly, such that 
\equh\label{eq:H}
H(\d s) = h(s)\d s \qmwith h(s) =\left\{
\begin{array}{l@{\mif}l}
h_0(s) & s\in[0,1/2]\\
h_1(s) & s\in[1/2,1]
\end{array}
\right.
\eque
yields a probability measure on $[0,1]$ with mean $1/2$. (For a simple choice, one can consider $H$ to be symmetric by setting $h_0(s) = h_1(1-s)$ for $s\in[0,1/2]$.)
Thus, for $(X,Y)$ corresponding to the spectral representation~\eqref{eq:01} with $H$ given in~\eqref{eq:H}, 
the extremes of $X$ are compressed by the extremes of $Y$, resulting in an exponentially light tail of $X/Y$, which is not regularly varying.
\end{Example}

The next example is the spectrally discrete bivariate 1-Fr\'echet random vectors.
\begin{Example}\label{example:SD}
A bivariate 1-Fr\'echet random vector $(X,Y)$ is {\it spectrally discrete}, if it has the following representation
\equh\label{eq:XYSD}
(X,Y)\eqd \bpp{\bigvee_{i=1}^m a_iZ_i, \bigvee_{i=1}^m b_iZ_i}\,,
\eque
where  $Z_1,\dots,Z_m$ are i.i.d.~standard 1-Fr\'echet random variables, $a_i\geq 0, b_i\geq 0, i = 1,\dots, m$. 
This model corresponds, in~\eqref{eq:XY}, to choose $S = \{1,\dots,m\}$, $\mu$ as the counting measure on $S$, $f(i) = a_i$ and $g(i) = b_i$.

Consider~\eqref{eq:XYSD} with $a_i/b_i$ strictly increasing.
If $a_m/b_m<\infty$, then clearly we have $\limm\bveein{X_i}/{Y_i} = {a_m}/{b_m}$ almost surely. Otherwise, if $a_m/b_m = \infty$ ($b_m = 0$) then $\gamma(t) = a_m/t$ for $t>a_{m-1}/b_{m-1}$. Therefore, Theorem~\ref{thm:RV} implies $X/Y\in RV_{-1}$.  That is, the ratio of spectrally discrete vectors is either bounded, or with tail index 1.
\end{Example}
Below, we calculate the ratio tail indices of two popular bivariate Fr\'echet distributions, the {\it logistic model} and the {\it mixed model}. These models have been studied carefully and applied widely to many real data analyses. See \cite{beirlant04statistics} and the references therein for detailed results on these models and their modifications.
\begin{Example}[The logistic model (\cite{gumbel60bivariate})]\label{example:logistic}
Consider the bivariate 1-Fr\'echet distribution
\equh\label{eq:logistic}
\proba(X\leq x,Y\leq y) = \exp\bbb{-\bpp{\frac1{x^\alpha}+\frac1{y^\alpha}}^{1/\alpha}}\,, x>0, y>0
\eque
for $\alpha\in[1,\infty]$. 
Here, the parameter $\alpha$ characterizes the dependence between $X$ and $Y$. The random variables $X$ and $Y$ are independent, if $\alpha$ = 1, and they are fully dependent, i.e., $\proba(X = Y) = 1$, if $\alpha = \infty$.

Example 5.13 in \cite{resnick87extreme} showed that this distribution has a representation~\eqref{eq:XY} with, when $\alpha>1$,
\[
f(s) = (\alpha-1)s^{\alpha-1}(1-s^\alpha)^{-1/\alpha}, g(s) = (\alpha-1)s^{\alpha-2}\,,
\]
and $(S,\mu) = ([0,1],Leb)$.
By straightforward calculation, we have
\[
E_t = \bccbb{s\in[0,1]:\frac{f(s)}{g(s)} > t} = \bccbb{s\in[0,1]:\frac{s^\alpha}{1-s^\alpha} > t^\alpha} = \Big(\frac t{(1+t^\alpha)^{1/\alpha}},1\Big]\,.
\]
Therefore, writing $\wt s = s^\alpha$,
\begin{multline*}
\gamma(t)=\frac{\nn f_{E_t}}
t = \frac1t\int_{\frac t{(1+t^\alpha)^{1/\alpha}}}^1f(s)\d s\\
= \frac1t\int_{\frac{t^\alpha}{1+t^\alpha}}^1\frac{\alpha-1}\alpha(1-\wt s)^{-1/\alpha}\d\wt s = \frac1t\bpp{\frac1{1+t^\alpha}}^{1-1/\alpha}\sim t^{-\alpha}\,.
\end{multline*}
Theorem~\ref{thm:RV} implies that the logistic model~\eqref{eq:logistic} has ratio tail index $\alpha$. This provides a probabilistic interpretation of $\alpha$. Another way to interpret the parameter $\alpha$ is provided by \cite{ledford98concomitant}. Therein, it is shown that $1/\alpha$ equals the limit probability that the component maxima do not occur at the same observation, i.e.,
\[
\limn \proba(\argmax_{i = 1,\dots,n}X_i = \argmax_{i = 1,\dots,n}Y_i) = 1-1/\alpha\,. 
\]
\end{Example}
\begin{Rem}[Estimation of the logistic model]
For the logistic model in Example~\ref{example:logistic}, curiously one may want to apply the Hill estimator (\cite{hill75simple}) on i.i.d.~copies $R_i\defe X_i/Y_i$ to estimate $\alpha$. Let $R_{(1,n)}\leq\cdots\leq R_{(n,n)}$ be the order statistics of $\{R_i\}_{i = 1,\dots, n}$. Then, the Hill estimator is defined by
\[
\what\gamma_H \equiv \what\gamma_H(k,n) \defe\frac1k\summ i1k \log R_{(i,n)} - \log R_{(k+1,n)}\,,
\]
depending on a threshold integer value $k$. This estimator is consistent in the sense that $\what\gamma_H\to\alpha\inv$ in probability, as long as $k/n\to 0$ and $k\to\infty$ as $n\to\infty$ (see e.g.~\cite{dehaan06extreme}~Theorem 3.2.2). Furthermore, the asymptotic normality of $\what\gamma_H$ is guaranteed by the {\it second-order condition} (see e.g.~\cite{dehaan06extreme}, Definition 2.3.1, Theorem~3.2.5). Indeed, since the explicit formula of $\proba(X/Y>t)$ is available, after some calculation we can show
\[
\sqrt k(\what\gamma_H-\alpha\inv)\weakto\calN(0,\alpha^{-2})\,,
\]
for $k = \left\lfloor n^{\beta/2}\right\rfloor$, $\beta\in(0,2/3)$. When $\beta = 2/3$, the Hill estimator has the optimal rate ($n^{1/3}$), but the limit will be a non-centered normal distribution. 

However, applying the Hill estimator here has little practical interest, 
since in this parametric model the maximum likelihood estimator works well with better rate ($n^{1/2}$): when $\alpha>1$, the estimation problem is regular and when $\alpha = 1$, the non-regular behavior of the maximum likelihood estimator has been addressed by \cite{tawn88bivariate}. \end{Rem}
\begin{Example}[The mixed model (\cite{gumbel62multivariate})]
Consider the bivariate 1-Fr\'echet distribution
\equh\label{eq:sibuya}
\proba(X\leq x,Y\leq y) = \exp\bbb{-\bpp{\frac1x+\frac1y-\frac k{x+y}}}, x>0, y>0
\eque
for $k\in[0,1]$. Here, $k = 0$ corresponds to the independent case, though $k = 1$ does not correspond to the full dependence case. By a similar calculation as in Example 5.13 in~\cite{resnick87extreme}, this distribution has a spectral representation~\eqref{eq:01} with
\[
H(\cdot) = k {\rm{Leb}}(\cdot) + \frac{1-k}2[\delta_{\{0\}}(\cdot) + \delta_{\{1\}}(\cdot)]\,,
\]
where $\delta_{\{0\}}$ and $\delta_{\{1\}}$ are unit point masses at $\{0\}$ and $\{1\}$, respectively. Straightforward calculation shows
\[
\gamma(t)= \frac{\nn f_{E_t}}t = \bccbb{{1-k}+k\bbb{1-\bpp{\frac t{1+t}}^2}}\frac1t\,.
\]
Therefore,
\[
\gamma(t)\sim\left\{
\begin{array}{l@{\qmif}l}
(1-k)t\inv & 0<k<1\\
2t^{-2} & k=1\,.
\end{array}
\right.
\]
When $k=1$, the ratio tail index equals $2$, and
\[
\frac1{n^{1/2}}\bveein\frac{X_i}{Y_i}\weakto2^{-1/2}\zeta_{2}\,,
\] 
where $\zeta_2$ is standard 2-Fr\'echet. When $k\in(0,1)$, the ratio tail index equals $1$, and
\[
\frac1{n}\bveein\frac{X_i}{Y_i}\weakto\frac1{1-k}\zeta_{1}\,,
\] 
where $\zeta_1$ is standard 1-Fr\'echet. The asymptotic behavior of the maxima of the ratios changes dramatically at $k=1$.
\end{Example}

\section{The Quotient Correlation and Independence Test}\label{sec:quotientCorrelation}
In this section, we apply our results on the ratio tail index to study the asymptotic behavior of a test statistic recently proposed by \cite{zhang08quotient}. Therein, Zhang proposed the gamma test for testing the independence between two random variables, based on their ratios. 
A similar test was also proposed aiming at testing the asymptotic independence (tail independence). The asymptotic behavior of the test statistics proposed have so far been studied only for independent Fr\'echet random variables. Here, we establish asymptotic results for jointly Fr\'echet random variables with arbitrary dependence structure. We show that, in most of the dependence cases, the power of the hypothesis test goes to one as the sample size increases to infinity.

\cite{zhang08quotient} essentially focused on the test statistics of the form
\equh\label{eq:qun}
q_{u,n} \defe\frac{\max_{1\leq i\leq n}( X_i(u)/ Y_i(u)) + \max_{1\leq i\leq n}( Y_i(u)/ X_i(u)) - 2}{\max_{1\leq i\leq n}( X_i(u)/ Y_i(u))\times \max_{1\leq i\leq n}( Y_i(u)/ X_i(u)) - 1}\,,
\eque
where $(X_i(u),Y_i(u))$ are i.i.d.~copies of $(X(u),Y(u))$ (recall~\eqref{eq:u}). When $u=0$, $q_n \defe q_{0,n}$ is called the {\it quotient correlation coefficient} and when $u>0$, $q_{u,n}$ is called the {\it tail quotient correlation coefficient}.

When $X$ and $Y$ are independent, \cite{zhang08quotient} showed that 
\[
\bpp{\frac1n\bveein\frac{ X_i(u)}{ Y_i(u)},\frac1n\bveein\frac{ Y_i(u)}{ X_i(u)}}\weakto\bpp{(1-\e^{-1/u})\zeta_1\topp1,(1-\e^{-1/u})\zeta_1\topp2}\,,
\]
and 
\equh\label{eq:nqun}
nq_{u,n}\weakto(1-\e^{-1/u})\inv\bpp{\frac1{\zeta_1\topp1}+\frac1{\zeta_1\topp2}}\eqd \Gamma(2,(1-\e^{-1/u})\inv)\,,
\eque
where $\zeta_1\topp1$ and $\zeta_1\topp2$ are independent standard 1-Fr\'echet random variables. Here $\Gamma(k,\theta)$ stands for the gamma distribution, which equals the distribution of the sum of $k$ independent exponential random variables with mean $\theta$. Recall that the inverse of a standard 1-Fr\'echet random variable has standard exponential distribution.

Based on~\eqref{eq:nqun} with $u = 0$, a hypothesis testing for independence was designed in~\cite{zhang08quotient} with the following null and alternative hypothesis:
\equh\label{eq:test}
H_0: \mbox{$X$ and $Y$ are independent}\mand H_1:\mbox{ $X$ and $Y$ are dependent.}
\eque
The test statistic $nq_{n}$ then has gamma limit distribution as~\eqref{eq:nqun} under the null hypothesis, the test~\eqref{eq:test} is thus referred to as the {\it gamma test}.
When $u > 0$, a similar hypothesis test was designed for testing asymptotic independence. 

In this section, we address the asymptotic behavior of $q_{u,n}$ for dependent bivariate 1-Fr\'echet random vectors. 
Recall again that for jointly Fr\'echet random variables, the independence and the asymptotic independence are equivalent. Therefore, we will focus on the independence test~\eqref{eq:test} and in particular $q_n = q_{0,n}$.
This essentially requires to investigate the limit of the joint distributions
\[
\bpp{\frac1{\kappa^+_n}\bveein \frac{X_i(u)}{Y_i(u)}, \frac1{\kappa^-_n}\bveein \frac{Y_i(u)}{X_i(u)}}
\]
with $u = 0$ as $n\to\infty$ for some suitable sequence $\indn{\kappa^\pm}$. However, all the asymptotic limits in this section would be the same for all $u>0$, up to multiplicative constants depending on $u$. We choose $u=0$ also for the sake of simplicity.

We first introduce some notations, since we need to deal with two groups of symbols, corresponding to $X/Y$ and $Y/X$ respectively. By default, the symbols with a sign `$+$' (`$-$' resp.) correspond to the ratio $X/Y$ ($Y/X$ resp.). In particular, for $(X,Y)$ as in~\eqref{eq:XY}, write $\gamma_+(t) = \gamma(t)$ as in~\eqref{eq:gamma} and 
\[
\gamma_-(t) = \frac{\int_Sg\ind_{\{g>tf\}}\d\mu}t\,.
\]
We have shown in Section~\ref{sec:examples} how to calculate the regular variation of $\gamma_+(t)$ for several models. Here $\gamma_-(t)$ can be treated similarly. In particular, if $f(s) = g(1-s)$ for $s\in[0,1]$ and $\mu$ is symmetric on $[0,1]$, or equivalently $H$ in~\eqref{eq:01} is symmetric on $[0,1]$, then $\gamma_+(t) = \gamma_-(t)$.
The following theorem and its corollary generalize Theorems 3.1 and 5.3 in \cite{zhang08quotient}.
\begin{Thm}\label{thm:1}
Consider $(X,Y)$ given by~\eqref{eq:XY}. Suppose $\mu(E_t)>0$ and $\mu(D_t)>0$ for all $t\in(0,\infty)$. Then, $\gamma_\pm\in RV_{-\alpha_\pm}$ for some $\alpha_+>0,\alpha_->0$, if and only if 
\equh\label{eq:asymptoticIndependence}
\bpp{\frac1{\kappa_n^+}\bveein\frac{ X_i}{ Y_i},\frac1{\kappa_n^-}\bveein\frac{ Y_i}{ X_i}}\weakto(\zeta_{\alpha_+},\zeta_{\alpha_-})\,,
\eque
where $\kappa_n^\pm\sim (1/\gamma_\pm)^\leftarrow(n)$ and $\zeta_{\alpha_\pm}$ are independent standard $\alpha_\pm$-Fr\'echet random variables, respectively. When $\alpha_+ = \alpha_- = \alpha$, $\zeta_{\alpha_+}$ and $\zeta_{\alpha_-}$ in~\eqref{eq:asymptoticIndependence} are interpreted as two independent standard $\alpha$-Fr\'echet random variables.
\end{Thm}
\begin{proof}
The `if' part follows from Theorem~\ref{thm:RV}. We show the `only if' part.
Write $R^+ =  X/Y$ and $R^- =  Y/X$.
Observe that, for $t_+, t_->1$,
\[
\proba(R^+\leq t_+, R^-\leq t_-) = 1 - \proba(R^+> t_+) - \proba(R^-> t_-)\,.
\]
Therefore, for all $t_+, t_->0$ and $n$ such that $\kappa_n^+t_+>1$ and $\kappa_n^-t_->1$,
\equh\label{eq:log}
 \log\proba\bpp{\frac1{\kappa_n^+}\bveein\frac{ X_i}{ Y_i}\leq t_+, \frac1{\kappa_n^-}\bveein\frac{ Y_i}{ X_i}\leq t_-}
\sim -n\bb{\proba(R^+>t_+\kappa_n^+) + \proba(R^->t_-\kappa_n^-)}
\,.
\eque
By definition of $\kappa_n^\pm$, we have
\[
n\proba(R^+>t_+\kappa_n^+)\sim\frac{\proba(R^+>t_+\kappa_n^+)}{\proba(R^+>\kappa_n^+)}\sim t_+^{-\alpha_+}
\]
and similarly $n\proba(R^->t_-\kappa_n^-)\sim t_-^{-\alpha_-}$. Therefore,
\begin{multline*}
\frac{n[\proba(R^+>t_+\kappa_n^+)+\proba(R^->t_-\kappa_n^-)]}{t_+^{-\alpha_+}+t_-^{-\alpha_-}}\\
= \frac{n\proba(R^+>t_+\kappa_n^+)}{t_+^{-\alpha_+}}\frac{t_+^{-\alpha_+}}{t_+^{-\alpha_+}+t_-^{-\alpha_-}} + \frac{n\proba(R^->t_+\kappa_n^-)}{t_-^{-\alpha_-}}\frac{t_-^{-\alpha_-}}{t_+^{-\alpha_+}+t_-^{-\alpha_-}}\to 1
\end{multline*}
as $n\to\infty$. Combined with~\eqref{eq:log} we have thus proved~\eqref{eq:asymptoticIndependence}.\qqed
\end{proof}
As a corollary, we establish the asymptotic behavior of $q_{n}$. The test statistic $q_n$ has different asymptotic limits, depending on different regular-variation type behaviors of $\gamma_\pm$. We write $a_n = o(b_n)$, if $\limn a_n/b_n = 0$. 
\begin{Coro}\label{coro:power}
Suppose that the assumptions in Theorem~\ref{thm:1} hold. Consider $\kappa_n^\pm\sim(1/\gamma_\pm)^\leftarrow(n)$. Then, 
\eqnhn\label{eq:limit}
\kappa_n^+q_n\weakto\zeta_{\alpha_+}\inv & \mif & \kappa_n^+=o(\kappa_n^-)\,,\nonumber\\
\kappa_n^-q_n\weakto\zeta_{\alpha_-}\inv & \mif & \kappa_n^-=o(\kappa_n^+)\,,\nonumber\\
\kappa_n^+q_n\weakto\frac1{\zeta_{\alpha_+}}+\frac C{\zeta_{\alpha_-}} & \mif & \kappa_n^+\sim C\kappa_n^-\mbox{ for some } C>0\,.\label{eq:kappanqn}
\eqnen
with $\zeta_{\alpha_\pm}$ as in Theorem~\ref{thm:1}. Note that $\alpha_+>\alpha_-$ implies $\kappa_n^+ = o(\kappa_n^-)$, and $\kappa_n^+\sim C\kappa_n^-$ implies $\alpha_+=\alpha_-$.
\end{Coro}
\begin{proof}
We only prove~\eqref{eq:kappanqn}. The proofs for the other two cases are similar. For the sake of simplicity, write 
\equh\label{eq:Rn}
R_n^+ = \bveein \frac{X_i}{Y_i},\quad R_n^- = \bveein \frac{Y_i}{X_i}
\eque
and $\what R_n^\pm = R_n^\pm/\kappa_n^+$. Then, by~\eqref{eq:qun}, we have
\equh\label{eq:kappan+qn}
\kappa_n^+q_n = \frac{\what R_n^+ + \what R_n^- - 2/\kappa_n^+}{\what R_n^+ \what R_n^- - 1/(\kappa_n^+)^2}\,.
\eque
Now, Theorem~\ref{thm:1} implies $(\what R_n^+,\what R_n^-)\weakto (\zeta_{\alpha_+},C\inv\zeta_{\alpha_-})$. By the continuous mapping theorem, we have proved~\eqref{eq:kappanqn}.\qqed
\end{proof}
Corollary~\ref{coro:power} provides theoretical support for the gamma test. Indeed, it shows that as long as $\kappa_n^\pm\sim(1/\gamma_\pm)^\leftarrow(n) = o(n)$, then $nq_n$ explodes quickly as $n\to\infty$. This means in this case, the gamma test rejects the null hypothesis with probability going to one as the sample size increases to infinity. Furthermore, the following example shows that when $\kappa_n^\pm\sim C_\pm n$ for some constants $C_\pm>0$, the gamma test still performs reasonably well. In effect, this is indeed the `worst' case that the gamma test could encounter, provided the ratio tail index exists.
\begin{Example}\label{example:power}
Recall the model~\eqref{eq:rho} considered in Example~\ref{example:patho}:
\[
(X,Y) \eqd ((1-\rho)Z_2\vee \rho Z_3,\rho Z_1\vee (1-\rho)Z_2)\,.
\]
In this case, $\gamma_\pm(t) = \rho/t$ for $t>1$, and the tail dependence index $\lambda = 1-\rho$. Corollary~\ref{coro:power} yields
\[
nq_{n}\weakto \frac1{\rho}\bpp{\frac1{\zeta_1\topp1}+\frac1{\zeta_1\topp2}}\eqd \Gamma(2,\rho\inv)\,.
\]
Now, observe that in the strong tail dependence case, i.e., $\rho$ is close to 0, the power of the test is high. Consider a test of level $0.05$. Let $q_\rho(\beta), \beta\in[0,1]$ denote the lower $\beta$ quantile of the distribution of $\Gamma(2,\rho\inv)$. For this model, the power converges to
\equh\label{eq:limitPower}
1 - q\inv_\rho(q_1(0.95))\,.
\eque
Figure~\ref{fig:power} illustrates the power of the test as a function of $\rho\in[0,1]$. We see that the test performs reasonably well as long as the tail dependence index $\lambda = 1-\rho$ is not too small. In addition, we also observe that the test statistic $nq_{n}$ converges quickly. For $n = 20$, the power of the test is already close to the limit one in~\eqref{eq:limitPower}.
\begin{figure}[ht!]
\begin{center}
\includegraphics[width = 0.8\textwidth]{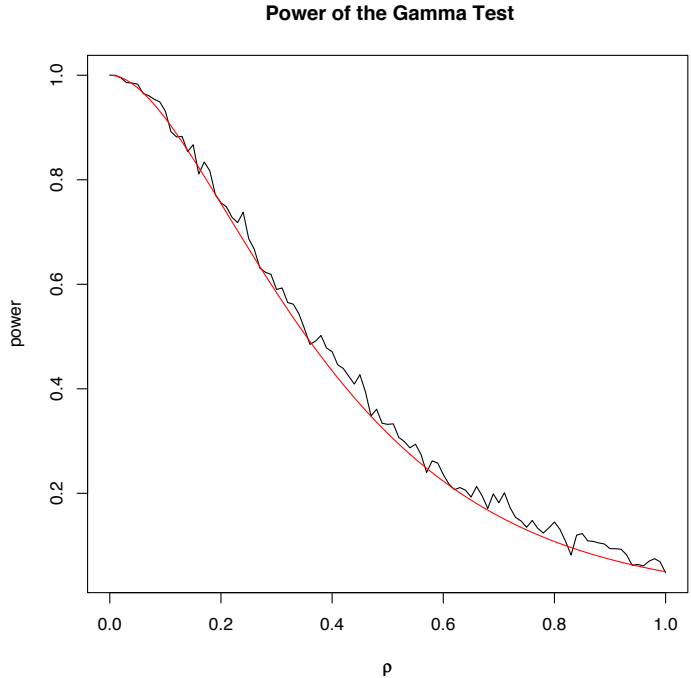}
\end{center}
\caption{\label{fig:power}Power of the gamma test (solid line) in the `worst scenario' (Example~\ref{example:power}) for selected $\rho$'s with $n = 20$, based on 1000 simulations. The dashed line corresponds to the limit power calculated in~\eqref{eq:limitPower}.}
\end{figure}
\begin{Rem}
In this experiment, we used a {\em modified gamma test} instead of the original one. Namely, instead of $q_n = q_{0,n}$ defined as in~\eqref{eq:qun}, we used a slightly different statistic $\what q_{n} \defe R_n^+R_n^-/(R_n^+ + R_n^-)$ with $R_n^\pm$ as in~\eqref{eq:Rn}. We do not use the original test statistic $q_{n}$ because, in this example when $\rho$ is small and the sample size $n$ is small, with very high probability the $q_{0,n}$ would equal $0/0$, which is not well defined. 

One can show that $nq_n$ and $n\what q_n$ have the same asymptotic distribution. Indeed, if we replace $q_n$ by $\what q_n$ in Corollary~\ref{coro:power}, all the statements remain valid. The proof will be the same, except that we have $\kappa_n^+\what q_n = ({\what R_n^++\what R_n^-})/(\what R_n^+\what R_n^-)$
instead of~\eqref{eq:kappan+qn}.
\end{Rem}
\end{Example}
\begin{Rem}\label{rem:power}
The model~\eqref{eq:rho} represents the `worst' case that the gamma test can encounter.
This is the case when the power of the gamma test is low, but the tail dependence is the strongest possible. 

To design such a scenario, we need $nq_{n}$ to converge and $(X,Y)$ to have the largest tail dependence index possible. To have $nq_n$ to converge, by Corollary~\ref{coro:power} we must have $\kappa_n^\pm\sim C_\pm n$, which is equivalent to the fact that $H$ has point masses at $\{0\}$ and $\{1\}$. Now, the tail dependence index is maximized by concentrating all the measure of $H$ on $(0,1)$ at $\{1/2\}$.

In this example, the gamma test performs poorly when $\rho$ is close to 1. However, one should not expect any independent test to perform well, as $(X,Y)$ as in~\eqref{eq:rho} can be seen as a pair of independent random variables $(\rho Z_1,\rho Z_3)$, slightly perturbed by $(1-\rho)Z_2$ (via the `max'($\vee$) operation).

\end{Rem}
In \cite{zhang08quotient}, other versions of $q_n$ were also proposed, aiming at dealing with arbitrary bivariate distributions. In principle, the data needs to be first transformed to have 1-Fr\'echet marginals. It is an intriguing problem to characterize how the dependence structure changes during such a transformation. 
The investigation along this line is out of the scope of this paper. 
\begin{Rem}\label{rem:transformation}
Observe that for continuous random variables $X$ and $Y$ with c.d.f.~$F_X$ and $F_Y$, the joint distribution of $(-1/\log F_X(X), -1/\log F_Y(Y))$ will have marginal standard 1-Fr\'echet distributions. However, most of the time this transformation does not lead to a bivariate 1-Fr\'echet distribution.
For studies on these distributions, see e.g.~\cite{ledford96statistics,ledford97modelling}. 
\end{Rem}

\section{Proof of Proposition~\ref{prop:1}}\label{sec:proofs}
We borrow a result from~\cite{weintraub91sample}. Weintraub's work is based on the min-stable distribution (processes), which can be equivalently transformed into our setting. Namely, if $(V,W)$ has a min-stable distribution with spectral functions $f$ and $g\in L^1(S,\mu)$ according to~\cite{weintraub91sample} (see also~\cite{dehaan86stationary}), then $(X,Y) \defe (1/V,1/W)$ is jointly 1-Fr\'echet with representation in~\eqref{eq:XY}.

Now, Lemma 3.4 in \cite{weintraub91sample} becomes, for $(X,Y)$ as in~\eqref{eq:XY},
\eqnhn
& & \proba(X\leq x\mid Y = y) \nonumber\\
\eqnhspace = \exp\bbb{-\int_S\bpp{\frac{f(s)}x-\frac{g(s)}y}\ind_{\ccbb{\frac{f(s)}{g(s)}> \frac xy}}\mu(\d s)}\int_Sg(s)\ind_{\ccbb{\frac{f(s)}{g(s)}\leq \frac xy}}\mu(\d s)\nonumber\\
\eqnhspace = \nn g_{D_t}\exp\bpp{-x\inv\nn f_{E_{x/y}}+y\inv\nn g_{E_{x/y}}}\,.\label{eq:weintraub}
\eqnen
Then, letting $\proba_Y$ denote the distribution of $Y$,
\eqnh
\proba(X/Y\leq t, Y>u) & = & \int_u^\infty\proba(X\leq tr\mid Y = r)\proba_Y(\d r)\\
& = & \int_u^\infty \proba(X\leq tr\mid Y = r)r^{-2}\exp(-r\inv)\d r\\
& = & \nn g_{D_t}\int_u^\infty r^{-2}\exp\bpp{-\frac{1+t\inv\nn f_{E_t} - \nn g_{E_t}}r}\d r\,.
\eqne
Observe that $1- \nn g_{E_t} = \nn g_{D_t}$. We have thus obtained~\eqref{eq:prop1}. The proof is complete.
\begin{Rem}
Careful readers may find the difference between our definition of $D_t$ in~\eqref{eq:Dt} and the one in~\cite{weintraub91sample}, Lemma 3.4 (where different symbols are used). Therein, $D_t = \{s\in S:f(s)/g(s) < t\}$, with the `$\leq$' in~\eqref{eq:Dt} replaced by `$<$'. After reading the two paragraphs of the proof of Lemma 3.4, one should see that the correct definition of $D_t$ is as in~\eqref{eq:Dt} with `$\leq$'. See also the proof of Lemma 3.5 in~\cite{weintraub91sample}, where Lemma 3.4 was applied with `$\leq$'.

An alternative way to see quickly it is not correct to choose the definition with `$<$' is given next. Note that, using `$<$' instead of `$\leq$' would only lead to a different formula when $\nn g_{D_t}$ (or $\nn f_{E_t}$, resp.) is discontinuous at some $t_0>0$. 
Consider
\[
F(t) \defe \proba(X/Y\leq t) = \proba(X/Y\leq t, Y>0) = \bpp{1+\frac {\nn f_{E_t}}{t\nn g_{D_t}}}\inv\,,
\]
which has jumps at $t_0$ such that $\mu\{s:f(s)/g(s) = t_0\}>0$. In this case, the definition with `$<$' would cause $F(t)$ to be right-discontinuous at $t_0$. Indeed, both $\nn f_{E_t}=\int_{E_t} f\d\mu$ and $\nn g_{D_t}=\int_{D_t} g\d\mu$ would be right-discontinuous at $t_0$. But $F(t)$, as a cumulative distribution function, should be right-continuous, which is a contradiction.

\end{Rem}
\section{Discussion}\label{sec:discussion}
Recently, several extensions of the notion of regular variation have been introduced. The main motivation behind them is to study in more details the asymptotic behaviors that are not captured by the standard multivariate regular variation. 
In particular,~\cite{resnick02hidden} introduced the notion of hidden regular variation, characterizing the dependence structure when components of the random vectors are asymptotically independent but not independent;~\cite{maulik02asymptotic} investigated the regular variation of products of random variables, and applied the result to model the network traffic.
We discuss our results from these perspectives. 
\medskip\\
{\it Hidden regular variation.} 
Recall the definition of regular variation in the language of vague convergence~\eqref{eq:VC}. When the exponent measure $\nu$ concentrates on the two axes, the two random variables $V$ and $W$ are asymptotically independent. In this case, 
$\nu$ does not provide useful information to characterize different dependence structures for asymptotically independent random variables. More sophisticated models are needed.

This problem was first investigated by \cite{ledford96statistics,ledford97modelling}, and their models later on were generalized under the framework of hidden regular variation by \cite{resnick02hidden,resnick07heavy,resnick08multivariate}. Hidden regular variation is present when the vague convergence~\eqref{eq:VC} holds with $\nu$ concentrated on the axes, i.e., $V$ and $W$ are asymptotically independent, and in addition there exists $\kappa_t^o = o(\kappa_t)$, such that 
\[
t\proba\bbb{\bpp{\frac V{\kappa_t^o},\frac W{\kappa_t^o}}\in\cdot} \Vagueto\nu^o(\cdot)
\]
in $M_+(\esp_0)$ with $\esp_0 = (0,\infty]\times (0,\infty]$. Intuitively, the notion of hidden regular variation involves normalizing the vector $(V,W)$ by sequences of constants of smaller order $(\kappa_t^o = o(\kappa_t)$) than required in~\eqref{eq:VC}. Thus, ignoring the two axes, certain dependence structure might appear in the limit. This dependence structure is not captured, therefore `hidden', when the `standard' rate $\kappa_t$ is taken.

In our case, when $(X,Y)$ is a bivariate 1-Fr\'echet random vector, we will see that $(V,W) = (X/Y,Y)$ are asymptotically independent, and there is no hidden regular variation. We first look at the convergence in form of 
\equh\label{eq:X/YY}
\nu_t(\cdot)\defe t\proba\bbb{\bpp{\frac{X/Y}{\kappa_t},\frac Yt}\in\cdot}\Vagueto\nu(\cdot)
\eque
in $M_+(\esp)$ with $\kappa_t\sim(1/\gamma)^\leftarrow(n)$. 
Here, we are in a slightly different situation, as the normalizing sequences $\kappa_t$ and $t$ are of different rates most of the time. In this case, we say $X/Y$ and $Y$ are asymptotically independent {\it with different rates}. The rates are the same, i.e., $\kappa_t\sim ct$ for some constant $c\in(0,\infty)$, if and only if $\int_Sf\ind_{\{g=0\}}\d\mu > 0$, by Corollary~\ref{coro:2}, part (i). Otherwise, $\kappa_t = o(t)$.

We first show below that $X/Y$ and $Y$ are asymptotically independent (with different rates if $(1/\gamma)^\leftarrow(t) = o(t)$). Indeed, for $r>0, y>0$, Proposition~\ref{prop:1} implies, writing $\wt t = \kappa_tr$ for simplicity,
\eqnhn
\nu_t([0,r]\times(y,\infty])& = & t\proba\bpp{\frac{X/Y}{\kappa_t}\leq r,\frac Yt>y} \nonumber\\
& = & t\frac{\nn g_{D_{\wt t}}}{\nn g_{D_{\wt t}} + {\wt t}\inv\nn f_{E_{\wt t}}}\bbb{1-\exp\bpp{-\frac{\nn g_{D_{\wt t}} + \wt t\inv\nn f_{E_{\wt t}}}{ty}}}\nonumber\\
& = & \frac{\nn g_{D_{\wt t}}}y + o(1) = \frac 1y + o(1)\,,\mmas t\to\infty\,.\label{eq:t}
\eqnen
It follows that for all rectangles $E = [r_1,y_1)\times[r_2,y_2)$ in $\mathbb E_0$, $\limt\nu_t(E) = 0$. 
Recall also that by~\eqref{eq:convergence} (with $u=0$), $t\proba(X/Y/\kappa_t>r)\sim r^{-\alpha}$. This together with~\eqref{eq:t} yields the following result.
\begin{Prop}
Let $(X,Y)$ be as in~\eqref{eq:XY} and suppose~\eqref{eq:RV} holds. Then, for $\kappa_t\sim(1/\gamma)^\leftarrow(t)$, the vague convergence~\eqref{eq:X/YY} holds in $M_+(\esp)$ with $\nu$ concentrated on the two axes. In other words,
\[
\bpp{\frac1{\kappa_n}\bveein\frac{X_i}{Y_i}, \frac1n \bveein Y_i}\weakto (\zeta_\alpha\topp1,\zeta_1\topp2)\,,
\]
where $\zeta_\alpha\topp1$ and $\zeta_1\topp2$ are independent standard Fr\'echet random variables with indices $\alpha$ and 1, respectively.
\end{Prop}
Now, we examine the hidden regular variation of this model. Suppose that there exist sequences $\kappa_t^o = o((1/\gamma)^\leftarrow(t))$ and $\iota_t^o = o(t)$, such that
\equh\label{eq:nuto}
\nu_t^o(\cdot)\defe t\proba\bbb{\bpp{\frac{X/Y}{\kappa_t^o},\frac Y{\iota^o_t}}\in\cdot} \vagueto\nu^o(\cdot)
\eque
in $M_+(\mathbb E_0)$ for some non-zero $\nu^o\in M_+(\mathbb E_0)$, i.e., $(X/Y, Y)$ is hidden-regularly varying. Then, by a similar calculation as in~\eqref{eq:t}, we have
\[
\nu_t^o((0,r]\times(y,\infty]) = \frac t{\iota_t^oy}\nn g_{D_{\kappa_t^or}}(1+o(1)).
\]
We see that $\nu_t^o$ cannot converge in $\mathbb E_0$. Therefore, we have proved the following.
\begin{Prop}
There is no hidden regular variation for $(X/Y, Y)$.
\end{Prop}
\noindent{\it Regular variation of products.}
\cite{maulik02asymptotic} proposed different generalizations of~\eqref{eq:VC}, for the purpose of characterizing the regular variation of the product of random variables. Their main results focus on two different situations.

First, if $V\in RV_{-\alpha}$ and 
\equh\label{eq:AI}
t\proba\bbb{\bpp{\frac V{\kappa_t},W}\in\cdot}\Vagueto(\nu_\alpha\times G)(\cdot)
\eque
in $M_+(\mathbb D), \mathbb D = (0,\infty]\times[0,\infty]$, where $\nu_\alpha((x,\infty]) = x^{-\alpha}, x>0$ and $G$ is a probability measure with $G(0,\infty) = 1$. Then,
\[
t\proba\bbb{\frac{(V,VW)}{\kappa_t}\in\cdot}\Vagueto\nu(\cdot)
\]
in $M_+(\mathbb D)$ with $\nu\in M_+(\mathbb D)$, determined by $\nu_\alpha$ and $G$ (\cite{maulik02asymptotic}, Theorem 2.1). This case can be seen as a generalization of Breiman's theorem, since $VW$ has an equivalent tail as $V$. Indeed,~\eqref{eq:AI} is referred to as the (new) definition of asymptotic independence of $V$ and $W$ therein.

Second, suppose $V\in RV_{-\alpha_V}, W\in RV_{-\alpha_W}$ for some $\alpha_V,\alpha_W>0$ and $V$ and $W$ are asymptotically {\it dependent} with different rates, or equivalently,
\[
t\proba\bbb{\bpp{\frac V{\kappa_t},\frac W{\iota_t}}\in\cdot} \Vagueto \nu(\cdot)
\]
in $M_+(\esp_0)$ with $\nu((0,\infty]^2) > 0$. Then, $VW\in RV_{-({\alpha_V\alpha_W})/({\alpha_V+\alpha_W})}$.
In this situation, the multiplication with a random variable changes the tail behavior. 

Our asymptotic result on the bivariate Fr\'echet random variables (Theorem~\ref{thm:RV}, with $V = X$ and $W = 1/Y$), does not fall into any of these situations. Indeed, if~\eqref{eq:RV} holds with $\alpha>1$, then $X/Y$ has a lighter tail than $X$, which differs from the first situation; on the other hand, $1/Y$ always has exponential distribution, which is not regularly varying, whence we are not in the second situation either. Therefore, the following question arises.
\begin{Ques}\label{ques:1}
Let $V$ and $W$ be two nonnegative random variables. Suppose $V\in RV_{-\alpha}$ (but $W$ may not be regularly varying). Provide a sufficient condition on $V$ and $W$ such that $VW$ is regularly varying with tail index $\beta>\alpha>0$.
\end{Ques}
\ifthenelse{\boolean{format}}{
\begin{acknowledgements}
The author was grateful to Stilian Stoev for his careful reading of an early version of the paper, as well as many inspiring and helpful comments and suggestions. The author would also like to thank the Editor J\"urg H\"usler and two anonymous referees for their helpful suggestions and comments, which significantly improved the paper.
\end{acknowledgements}}
{\vspace{0.2in}
\noindent{\bf Acknowledgements}
The author was grateful to Stilian Stoev for his careful reading of an early version of the paper, as well as many inspiring and helpful comments and suggestions. The author would also like to thank the Editor J\"urg H\"usler and two anonymous referees for their helpful suggestions and comments, which significantly improved the paper. The author was partially supported by the National Science Foundation grant DMS--0806094 at the University of Michigan. 
}
\def\doi#1{DOI: #1}

\def\cprime{$'$}

\end{document}

%% file: include_macros.tex
\newcommand{\ind}{{\bf 1}}
\newcommand{\proba}{\mathbb P}
\newcommand{\esp}{{\mathbb E}}

\newcommand{\defe}{\mathrel{\mathop:}=}

\newcommand{\inv}{^{-1}}

\newcommand{\argmax}{{\rm{argmax}}}

\newcommand{\calB}{{\cal B}}

\newcommand{\filF}{{\cal F}}

\newcommand{\calN}{{\cal N}}

\def\indn#1{\{#1_n\}_{n\in \mathbb N}}

\newcommand{\eqnh}{\begin{eqnarray*}}
\newcommand{\eqne}{\end{eqnarray*}}
\newcommand{\eqnhn}{\begin{eqnarray}}
\newcommand{\eqnen}{\end{eqnarray}}
\newcommand{\equh}{\begin{equation}}
\newcommand{\eque}{\end{equation}}

\def\summ#1#2#3{\sum_{#1 = #2}^{#3}}

\def\bveein{\bigvee_{i=1}^n}

\newcommand{\eqd}{\stackrel{\rm d}{=}}

\newcommand{\widebar}{\overline}

\def\topp#1{^{(#1)}}

\def\nn#1{{\left\|#1\right\|}}

\def\ccbb#1{\left\{#1\right\}}
\def\bccbb#1{\Big\{#1\Big\}}

\def\bpp#1{\Big(#1\Big)}

\def\bb#1{\left[#1\right]}

\def\bbb#1{\Big[#1\Big]}

\def\d{{\rm d}}
\def\e{{\rm e}}

\def\eqnhspace{& & \ \ \ \ }

\def\mand{\mbox{ and }}
\def\qmand{\quad\mbox{ and }\quad}

\def\qmwith{\quad\mbox{ with }\quad}
\def\mfa{\mbox{ for all }}

\def\mmas{\mbox{ as }}

\def\mif{\mbox{ if }}
\def\qmif{\quad\mbox{ if }\quad}

\def\adaptF#1{\{#1_t,\filF_t:0\leq t<\infty\}}

\def\weakto{\Rightarrow}

\def\vagueto{\stackrel{v}{\rightarrow}}
\def\Vagueto{\stackrel{v}{\longrightarrow}}

\def\limn{\lim_{n\to\infty}}
\def\limm{\lim_{m\to\infty}}
\def\limt{\lim_{t\to\infty}}

\def\wt#1{\widetilde{#1}}
\def\wb#1{\widebar{#1}}

\def\what#1{\widehat{#1}}

\def\ifhead#1#2{\left\{\begin{array}{#1@{\quad\mbox{ if }\quad}#2}}
\def\ifend{\end{array}\right.}
